\documentclass{amsart}
\usepackage{amsmath,amssymb}
\usepackage[all]{xy}

\theoremstyle{plain}
\newtheorem{introquestion}{Question}
\newtheorem{introtheorem}{Theorem}
\newtheorem{theorem}{Theorem}[section]
\newtheorem{proposition}[theorem]{Proposition}
\newtheorem{lemma}[theorem]{Lemma}

\newtheorem*{proposition*}{Proposition}

\theoremstyle{definition}
\newtheorem{definition}[theorem]{Definition}
\newtheorem{example}[theorem]{Example}

\theoremstyle{remark}
\newtheorem{remark}[theorem]{Remark}

\newcommand{\thmref}[1]{Theorem~\ref{#1}}

\newcommand{\lemref}[1]{Lemma~\ref{#1}}

\newcommand{\exref}[1]{Example~\ref{#1}}
\newcommand{\queref}[1]{Question~\ref{#1}}

\def\B{{\mathcal B}}
\def\D{{\mathcal D}}

\def\ad{{\mathrm{ad}}}

\def\Q{{\mathbb Q}}
\def\C{{\mathbb C}}

\def\map{\mathrm{map}}
\def\Der{\mathrm{Der}}
\def\Hom{\mathrm{Hom}}

\def\cat0{\mathrm{cat}_0}

\def\dim{\mathrm{dim}}

\def\aut{\mathrm{aut}_1}

\def\w{\mathrm{wt}}
\def\B{B\mathrm{aut}_1}

\begin{document}

\title[Realizing Spaces as Classifying Spaces]
{Realizing Spaces as Classifying Spaces}

\author{Gregory  Lupton}

\address{Department of Mathematics,
           Cleveland State University,
           Cleveland OH 44115}

\email{G.Lupton@csuohio.edu}

\author{Samuel Bruce Smith}

\address{Department of Mathematics,
   Saint Joseph's University,
   Philadelphia, PA 19131}

\email{smith@sju.edu}


\keywords{Classifying Space for Fibrations, Rational Homotopy Type, Derivations, Minimal Model, Finite H-Space}

\subjclass[2010]{Primary: 55P62 55R15; Secondary: 55P10}

\thanks{This work was partially supported by a grant from the Simons Foundation (\#209575 to Gregory Lupton).  The research was also supported through the program ``Research in Pairs'' by the Mathematisches Forschungsinstitut Oberwolfach in 2014}

\begin{abstract}Which spaces occur as a classifying space for fibrations with a given fibre?  We address this question in the  context of rational homotopy theory.  We construct an infinite  family of  finite complexes  realized (up to rational homotopy) as classifying spaces.     We also give several non-realization results, including the following:  the rational homotopy types of $\C P^2$ and $S^4$ are not realized as the classifying space of any simply connected, rational  space with finite-dimensional homotopy groups. 
\end{abstract}

\maketitle

\section{Introduction}

The  classification  theory for   fibrations with fibres equivalent to  a fixed  CW complex $X$ was developed in     a series of  seminal papers  \cite{St, Dold,  May}.    The result of this work is the existence  of a {\em classifying space},  written     $B\mathrm{aut} (X)$.  The space $B\mathrm{aut} (X)$ is     the base of a universal fibration with fibre $X$ setting up a  one-to-one correspondence between  fibre-homotopy types of fibrations $X \to E \to B$   and homotopy classes of maps $h \colon B \to B \mathrm{aut} (X).$
The   exuberant notation  for the classifying space  is accounted for by its  provenance:  up to homotopy type, the  space may be  obtained by applying the  Dold-Lashof classifying space construction  \cite{D-L59} to the monoid $\mathrm{aut} (X)$  of  all homotopy self-equivalences of $X$ (see  \cite{Fuchs}).  
Restricting to the sub-monoid $\aut (X) := \map(X, X; 1)$   gives  the universal cover $B \aut (X)$,  the classifying space for fibrations   $X \to E \to B$ with $B$ simply connected.

The space  $\B(X)$ was among the first geometric objects described in rational homotopy theory. Sullivan   gave a model for  
this simply connected classifying space in terms of the derivations of a Sullivan minimal model \cite[Sec.11]{Su}.  Schlessinger and Stasheff \cite{SS} constructed a second, equivalent model in terms of derivations of a Quillen model.  The following  is a  long-standing, open question  in rational homotopy theory  (see \cite[p.519]{F-H-T01}):

\begin{introquestion}\label{ques: Baut}  
Which simply connected rational homotopy types occur as $B \aut (X)$?
\end{introquestion}

Question \ref{ques: Baut} is often interpreted as a conjecture to the effect that {\em all} rational homotopy types occur as classifying spaces.  However, such a suggestion is perhaps best viewed as an admission that, except in restricted cases,  little is known about the possible  rational homotopy types  that may occur.      The  affirmed cases of the famous Halperin Conjecture  \cite[p.516]{F-H-T01}  imply that $\B(X)$ is a product of even-dimensional Eilenberg-Mac Lane spaces for certain formal spaces $X$ (see \cite{Meier, Sm}). Gatsinzi \cite{Gat95, Gat96, Gat98, Gat01}  obtained a variety of results showing that  the L-S category of $B \aut (X)$ is  infinite for certain classes of spaces.      Yamaguchi \cite{Yam05} identified  the possible elliptic  spaces $X$ for which $B \aut (X)$ is of the rational homotopy type of a (rank-one) Eilenberg-Mac Lane space. 

The published  results on the classifying space taken together reveal a significant common feature.   With the exception of the odd-dimensional sphere  $$S^{2n+1} \simeq_\Q  B \aut (K(\Q, 2n)),$$ all  rational homotopy types known to occur correspond to infinite-dimensional CW complexes.   In this paper,   we give a new family of finite complexes  realized (up to rational homotopy type) as $B \aut (X)$.   We also prove that   the   rational homotopy types of some finite complexes  cannot be   realized when  $X$  is restricted to have finite-dimensional rational homotopy groups.

As an overriding hypothesis, we assume      all spaces  $X$ appearing in this paper are {\em rational} spaces. That is, all spaces  satisfy $X = X_\Q$.   We further assume spaces $X$ are    nilpotent, usually simply connected,  and of finite type.         
 We introduce one further hypothesis that will facilitates our analysis.  We say   a space $X$ is  \emph{$\pi$-finite} if $X$ has only finitely many non-zero (rational) homotopy groups.  
In this case, $\B (X)$ is  also a $\pi$-finite rational space  (see Proposition \ref{lem: pi-finite top degree}, below). 

It is easy to prove that   
$B \aut (K(\Q^m, n-1)) = K(\Q^m, n)$
for $n \geq 2$ and $m\geq 1$.  It is natural then to attempt to realize a product of Eilenberg Mac Lane spaces with nonzero homotopy groups in two distinct degrees.      In Section \ref{sec:realize},  we prove the following: \begin{introtheorem}\label{thm: intro rank 2}
The following rational homotopy types occur as $B \aut (X)$ for some simply connected, $\pi$-finite, rational space $X$:
\begin{enumerate}
\item  $S^{2n+1} \times S^{4n+1}$, for $n \geq 1$ and $n$ odd; 
\item  $K(\Q, r) \times K(\Q, r + 4m +1)$ for $r \geq 2$ and $m \geq 1$.
\end{enumerate}
We may also take $m=0$ in (2), if we allow $X$ to be nilpotent (not simply connected).
\end{introtheorem}

In Section \ref{sec: non-realization}, we prove the following non-realization result:  \begin{introtheorem}\label{thm: no CP or S4}
The rational homotopy types of $\C P^2$ and $S^4$ are not realized as the classifying space of any simply connected, $\pi$-finite,  rational  space.
\end{introtheorem}
Theorem \ref{thm: no CP or S4} strikes a warning note, as regards \queref{ques: Baut}: It implies that to realize these simple rational types will require quite complicated spaces $X$, i.e., spaces with infinitely many non-zero homotopy groups. 
Also in Section \ref{sec: non-realization}, we  deduce that any  simply connected space  of dimension five that does not satisfy a certain structural condition---of which there are many examples---cannot be realized as the classifying space of any simply connected, $\pi$-finite space.  
Whereas all  results up to this point  are obtained by analysis of Sullivan's model for the classifying space $B \aut (X)$, we include one further result  using the Schlessinger-Stasheff model.

\section{Preliminaries in Rational Homotopy Theory} 
\label{sec:prelim}
In this section, we   establish notation  in rational homotopy theory and record some   facts we will use about  the classifying space $\B (X)$.  We then give   two examples, one  of a realization and the other  a non-realization result. We emphasize   again our overriding hypotheses that spaces $X$ introduced are assumed  to be {\em rational}.   
This assumption  allows  for a concise statement of results avoiding the various hypotheses required to rationalize classifying spaces.

A     nilpotent space $X$ of finite type admits  a  \emph{Sullivan minimal model} $  (\land(V), d)$, which is a differential graded (DG) algebra freely generated by a  connected rational vector space  $V$  of finite dimension in each degree.   The differential  $d$ satisfies the {\em minimality condition} $d(V) \subseteq \land ^{+} V \cdot \land^{+} V$.  More generally, a  fibration $X \to E \to B$ of nilpotent spaces with $B$ simply connected  corresponds to a Koszul-Sullivan extension ({\em KS-extension}).  This is  a  sequence of DG algebras $$(\land W, \delta)  \to  (\land W \otimes \land V, \D) \to (\land V, d),$$ in which
 $(\land W, \delta)$ and $(\land V,  d)$ are the minimal models for $B$ and $X$, respectively.  Furthermore, the DG algebra  $(\land W \otimes \land V, \D)$ is a model for $E$ but need not be minimal; the differential here satisfies   $\D(w) = \delta(w)$ for $w \in W$ while $\D(v) - d(v) \in \land^{+}W \cdot (\land W \otimes \land V)$ 
 for $v \in V.$   Our   references for rational homotopy theory  are \cite{Tanre, F-H-T01}. 
 
Sullivan's model for the classifying space $\B (X)$   is constructed in terms of derivations of the minimal model $(\land V, d)$ for $X$   \cite[Sec.11]{Su}.    Let   $(\Der( \land V), D)$ denote   the graded Lie algebra of negative-degree derivations of $\land V$.  That is,  $\theta \in \Der^n (\land V)$ reduces degrees by $n$ and satisfies the derivation law $\theta(\chi_1\chi_2) = \theta(\chi_1)\chi_2 + (-1)^{n|\chi_1|} \chi_1\theta(\chi_2)$ for $\chi_1, \chi_2 \in \land V.$   The bracket of two derivations is   $[\theta_1, \theta_2] = \theta_1 \circ \theta_2 - (-1)^{|\theta_1||\theta_2|}\theta_2 \circ \theta_1$   and the  differential $D$ is given by   $D(\theta) = [d,\theta]$ for $\theta \in \Der (\land V).$  The DG Lie algebra $(\Der(\land V), D)$ gives rise to a Quillen model for $\B(X)$ (see \cite[Ch.6]{Tanre} and \cite{Gat95}).    We  will only  need the following special case of this  result here:

\begin{theorem}  Let $X$ be nilpotent   space of finite type with Sullivan minimal model $(\land V, d)$. There is an  isomorphism  of graded Lie algebras
$$   \pi_*(\Omega B \aut (X))  \cong H_*(\Der(\land V)) $$  in positive degrees where the left-hand graded space  has the Samelson bracket.   
\end{theorem}
\begin{proof}
A direct proof for $X$ simply connected using the identity $\Omega \B (X) = \aut (X)$ is given  \cite[Th.1]{FLS}. 
The  argument given there requires only the existence of a Sullivan minimal model  for $X$ and so may be extended to the case $X$ is nilpotent.
\end{proof}

 \begin{proposition}\label{lem: pi-finite top degree}
Suppose $X$ is nilpotent and $\pi$-finite with 
$$\pi_i(X)  = \begin{cases} \Q^r\  \mathrm{some}\  r \geq 1 & i = N \\ 0 & i >  N.  \end{cases}$$
Then we have
$$\pi_i(B \aut X)  = \begin{cases} \Q^r & i = N+1 \\ 0 & i >  N+1.  \end{cases}$$
\end{proposition} 

\begin{proof}
By hypothesis,  $X$ has minimal model of form $\land V$ with $V$ non-zero only in degrees $\leq N$, and $V^N$ of dimension $r$.  It follows that  $\Der(\land V)$ is a graded vector space that is non-zero only in degrees $\leq N$.  Therefore, we have
$$\pi_{i+1}(B \aut X) \cong \pi_{i}(\Omega B \aut X) \cong H_{i}\big(\Der( \land V)\big) = 0$$
for $i > N$.

Furthermore, in degree $N$, for each $\theta \in \Hom(V^N, \Q)$, we obtain a derivation in $\Der( \land  V)$ of degree $N$ by setting $\theta(V^N) = 0$ and extending as a derivation.  Any such derivation is a $D$-cycle, since the elements of $V^N$---as the last stage of generators---do not occur in the differential of any other generators.  There are no non-zero boundaries of degree $N$, since $\Der(\land V)$ is zero in degree $N+1$ (and higher).  So the vector space $V^N$ persists to homology, and we have
$$\pi_{N+1}(B \aut (X))  \cong \pi_{N}(\Omega B \aut (X))  \cong H_{N}\big(\Der( \land V)\big) \cong \Hom(V^N,\Q).$$
\end{proof}
We next describe a situation in which we can be assured of a nontrivial fibration  $X \to E \to B$ and thus an essential classifying map.   We formulate the result in terms of derivations.  The proof uses  the  {\em Gottlieb group}   $G_*(E) \subseteq \pi_*(E)$.  We recall that $G_*(E)$  is  the image of the map induced on homotopy groups by the evaluation map $\omega \colon \aut (E) \to E$.  
    For $E$ a simply connected rational space with minimal model $(\land (W), d)$,   we have an identification $$G_*(E) =  \mathrm{Im} \{ H(\varepsilon) \colon H_*(\Der(\land  W)) \to \mathrm{Hom}(W, \Q)\} $$
Here $\varepsilon \colon \land W \to \Q$ is the augmentation and so    $H(\varepsilon)(\langle \theta \rangle)$ is the  restriction of the $D$-cycle  $\theta$ to the basis $W$  (see \cite[Th.3.5]{LS1}).   We prove:  \begin{proposition}  \label{lem:non-trivial}    
 Let $X$ be simply connected with Sullivan minimal model $(\land V, d)$.   Suppose  given a KS-fibration 
  $(\land(w_n), 0) \to (\land V \otimes \land(w_n), \D) \to (\land V, d)$
with $w_n$   of degree $n > 1$.   Suppose  the following conditions hold:
\begin{enumerate}
\item  $ (\land(w_n)  \otimes \land V, \D)$ is a minimal DG algebra and  
\item   any derivation $\theta \in \Der_n(\land(w_n) \otimes \land V )$ with  $\theta(w) \neq 0$ satisfies  $D(\theta) \neq 0.$  
\end{enumerate}
    Then there exists an essential map $K(\Q, n) \to \B(X).$  
  \end{proposition} 
\begin{proof}
The spatial realization of the given KS-fibration is a fibration of the form $X \to E \stackrel{p}{\to} K(\Q, n)$.  Our hypothesis (2)  implies that the dual of the basis vector $w_n$ is not in the Gottlieb group $G_*(E).$        Thus  $\mathrm{rank}(G_n(E)) \leq  \mathrm{rank}(G_n(X)).$ As regards the product, we have  $G_n(K(\Q, n) \times X) =   \Q \oplus G_n(X).$  Thus $E \not \simeq K(\Q, n) \times X$ and    the classifying map $K(\Q, n) \to \B (X)$ for $p$ is the needed essential map. 
\end{proof}  
We conclude  this section with two  simple examples.  We begin with a realization result for a rank-two H-space.

\begin{example} \label{ex:(8,5)} 
Suppose $X$ has minimal model $(\land(x_3, y_3, z_5, w_7), d)$, with subscripts denoting degrees and differential
$d(x) = 0$, $d(y) = 0$, $d(z) = xy$, and $d(w) = xz$.  Since $\land V$ is freely generated by $V$, any derivation in $\Der(\land V)$
may be specified by its effect on generators in $V$.  Then in positive degrees, a vector space basis for $\Der(\land V)$ may be displayed as follows:

\begin{center}
\begin{tabular}{c|c}
degree & generators\\
\hline
7 & $w^*$\\
\hline
5 & $z^*$\\
\hline
4 & $(w,x)$, $(w,y)$\\
\hline
3 & $x^*$, $y^*$\\
\hline
2 & $(z,y)$, $(z,x)$, $(w,z)$\\
\hline
1 & $(w,xy)$\\
\end{tabular}
\end{center}
Here we are using the notation $(w,x)$ for the derivation that sends $w$ to $x$ and all other generators to $0$, we have written $w^*$ for $(w, 1)$, and so-on.   Direct computation shows that the differential $D$  in $\Der(\land V)$ is given by
$$D(w^*) = 0, \qquad D(z^*) = -(w,x), \qquad D\big( (w,x)\big) = 0, \qquad D\big( (w,y)\big) = 0,$$
$$D(x^*) = (z,y) + (w, z), \qquad D(y^*) = -(z,x), $$
$$D\big( (z,y)\big) = -(w, xy), \qquad D\big( (z,x)\big) = 0, \qquad D\big( (w,z)\big) = (w, xy), $$
$$D\big( (w,xy)\big) = 0.$$
Then  the homology of $\Der(\land V)$  is of rank $1$ in degrees $7$ and $4$, and zero otherwise. Thus, 
$B \aut (X)$ has homotopy groups of rank $1$ in degrees $8$ and $5$.  It follows that we must have
$$B \aut (X) = K(\Q, 5) \times K(\Q, 8).$$
\end{example}

Our next example shows that not all rank-two rational H-spaces can be realized as $\B (X)$ for $X$ simply connected and $\pi$-finite. 

\begin{example} \label{ex:(3,4)} 
We show $K(\Q, 3) \times K(\Q, 4)$ cannot be so realized. 
For suppose $X$ is a simply connected, $\pi$-finite, rational space with $B\aut (X) = K(\Q, 3) \times K(\Q, 4)$.  Since $\pi_*(B\aut (X))$ is zero above degree $4$,  we can conclude  that $\pi_*(X) $ is concentrated in degrees $2$ and $3$.  Further, since we have $\pi_4(B\aut (X)) = \Q$, we must have  $\pi_3(X) = \Q$  by Proposition \ref{lem: pi-finite top degree}.   
Thus the minimal model   for $X$ takes the form $(\land(x_1, \ldots, x_r, y), d)$ with the $x_i$ in degree $2$, $y$ in degree $3$ and  $dx_i = 0$ (for degree reasons).   Proceeding as in \exref{ex:(8,5)}, we may write a vector space basis for  $\Der(\land V)$ in positive degrees as follows:  
\begin{center}
\begin{tabular}{c|c}
degree & generators\\
\hline
3 & $y^*$\\
\hline
2 & $x^*_1, \ldots, x^*_r$\\
\hline
1 & $(y, x_1), \ldots, (y, x_r)$\\
\end{tabular}
\end{center}
We see that $D(y^*) = 0$ and $D\big((y, x_i)\big) = 0.$ If  $\Q  = \pi_3(B \aut (X)) \cong H_2(\Der(\land V), D)$ the map $D \colon \Der^2 \land V \to \Der^1 \land V$ must have  kernel of dimension $1$ and so   image    of dimension $r-1$.  Thus $H_1(\Der(\land V), D) \cong \Q$, contradicting the fact that $\pi_2(B \aut (X)) \cong H_1(\Der(\land V), D) = 0.$ 
\end{example}

In fact,  we may realize $K(\Q, 3) \times K(\Q, 4)$ as the classifying space of a nilpotent (non-simply connected) space---see \thmref{thm: r r+4m+1}.  Examples \ref{ex:(8,5)} and \ref{ex:(3,4)}, along with the results in the next section, indicate the challenge faced in addressing Question \ref{ques: Baut}.  Even amongst rank-two H-spaces $K(\Q, m) \times K(\Q, n)$, it seems difficult  to predict simply  from the degrees $m$ and $n$   
whether or not---and if so, how---the rational homotopy type  can be realized as the classifying space of a $\pi$-finite complex.

\section{Rank-two  $H$-spaces realized as $B \aut (X)$}
\label{sec:realize} 
In this section, we make constructions that realize certain rank-two H-spaces as classifying spaces.    
We first prove  part (2) of \thmref{thm: intro rank 2} of the Introduction.

\begin{theorem}\label{thm: r r+4m+1}  
For each $r \geq 2$ and $m \geq 0$, there exists a  $\pi$-finite, rational space $X_{r,m}$ with 
$$ B\aut (X _{r, m}) = K(\Q, r) \times K(\Q, r + 4m +1).$$
If $m \geq 1$, then we may take $X_{r, m}$ to be simply connected.  If $m =0$, then we require that $X_{r, m}$  be nilpotent, non-simply connected.
\end{theorem}
\begin{proof}  We define the space $X_{r,m}$ in terms of a minimal model  $(\land(u_{2m+1}, v_{2m+r}, y_{4m+r}), d)$ with subscripts indicating degrees and nonzero differential  $d(y) = uv.$    The generators of $\Der(\land V)$ are given by the table:
\begin{center}
\begin{tabular}{c|c}
degree & generators\\
\hline
$4m+r$  & $y^*$\\
\hline
$2m+r$ & $v^*$\\
\hline
$2m+r-1$ & $(y, u)$\\
\hline
$2m+1$ & $u^*$ \\
\hline
$2m$ & $(y, v)$ \\
\hline
$r-1$ & $(v, u)$ \\
\end{tabular}
\end{center}
Note that we may have $r-1 < 2m$, as pictured, or it may fall in the range $2m \leq  r-1 < 2m + r-1$; this makes no difference to our calculation.
The only nonzero differentials are $$ D(v^*) =   \pm (y, u) \hbox{\, and \,} D(u^*) = \pm(y, v).$$
Thus $H_*(\Der(\land V), D)$ has rank $1$ in degrees $r-1$ and $4m +r$ and is trivial in all other degrees.  It follows that $\pi_i(\aut (X_{r, m})) \cong \Q$ for $i = r-1$ and $i = 4m+r$, and zero otherwise.  Hence, we  see that $B\aut (X_{r, m})$ has the correct homotopy groups.  When $4m+2$ is not a  multiple of $r$, or if $r$ is odd, this is sufficient to determine that the rational homotopy type of 
$B\aut (X_{r, m})$ is as asserted, since there is only one rational homotopy type with such homotopy groups.  

So suppose that $4m+2 = kr$ for some $k\geq1$, and that  $r$ is even.  Here, we must distinguish $B \aut (X_{r, m})$ from the space $Z$ with truncated polynomial cohomology $H^*(Z) = \land(z_r)/ \langle z_r^{k+1}\rangle$, where $z_r$ denotes a generator of (even) degree $r$.    We use Proposition \ref{lem:non-trivial} to do so.  Define a KS-extension
$$   (\land(z_r),  0) \to  (\land(z) \otimes (\land (u, v, y), \D) \to   (\land(u, v, y), d)$$
by setting $\D(u) = \D(z) = 0$, $\D(v) = uz$ and $\D(y) = d(y) = uv.$  Notice that, since $u^2 = 0$, we have $\D^2 = 0$. 
Clearly,  the DG algebra $(\land(z) \otimes \land (u, v, y),  \D)$ is minimal.  Observe that  $D(z^*)=  \pm (v,u).$
It follows easily that $\theta(z) \neq 0$ implies $D(\theta) \neq 0.$ We conclude there is  an essential map $h \colon K(\Q, r) \to B\aut (X_{r,m}).$  Since $ Z$ admits no such map---as is easy to see, for example, using minimal models---we must have $B\aut (X_{r, m}) = K(\Q, r) \times K(\Q, r + 4m +1)$ in this case also.   
 \end{proof}
 
\begin{remark}
We may describe the spaces $X_{r, m}$ of \thmref{thm: r r+4m+1} without reference to minimal models, as two-stage Postnikov pieces.  Namely, $X_{r, m}$ is the total space in a principal fibration $K(\Q, 4m+r) \to X_{r, m} \to   K(\Q, 2m+1) \times K(\Q, 2m+r)$, with $k$-invariant   
$K(\Q, 2m+1) \times K(\Q, 2m+r) \to K(\Q, 4m+r+1)$ that corresponds to  the non-zero cup-product $uv \in H^{4m+r+1}\big(K(\Q, 2m+1) \times K(\Q, 2m+r)\big)$,  with $u \in H^{2m+1}\big(K(\Q, 2m+1)\big)$ and $v \in H^{2m+r}\big(K(\Q, 2m+r) \big)$ generators.
\end{remark}

Now we complete \thmref{thm: intro rank 2} of the Introduction by defining simply connected, $\pi$-finite spaces $X_n$ with  $B \aut (X_n)$ of the rational homotopy type of $S^{2n+1} \times S^{4n+1}$, for each $n$ odd and $n \geq 1$.    We will write the details assuming that $n \geq 5$.  In the cases in which $n = 1$ or $3$, the details are very similar, with some minor differences due to the fact that, for these low-end cases, the degrees of some of the generators, or the differences between the degrees of some of the terms,  coincide (or become negative, in which case they may be set aside).  
 
We describe the space $X_n$ in terms of a minimal model.   The model has $6$ generators, and so $X_n$ is a $\pi$-finite space.  The following table gives a vector space basis for $\land V$ through the degrees of the highest generator.  Notice that, since $n$ is odd, and thus the degree of $v_1$ and $v_2$ is even, we must allow for powers of these generators. We will use this information to identify a basis for $\Der(\land V)$.

$$\begin{tabular}{c|c|c}
degree  & generator & decomposables\\ 
\hline
$4n$ & y &  \\
$3n+3$ &  & $v_1^3, v_1^2v_2, v_1v_2^2, v_2^3$ \\
$3n+1$ &  & $v_1w, v_2w$ \\
$3n$ & $u_1, u_2$ & \\
$2n+2$ &  & $v_1^2, v_1v_2, v_2^2$ \\
$2n$ & $w$ &  \\
$n+1$ & $v_1, v_2$ &  \\
\end{tabular}$$
The non-zero differentials in $\land V$ are defined to be
$$dy = u_1v_1 + u_2v_2, \quad du_1 = - v_2 w, \quad du_2 = v_1w.$$

\begin{theorem}\label{thm: 2n+1 4n+1}
With $X_n$ as above, we have $$B \aut (X_n) = K(\Q, 2n+1) \times K(\Q, 4n+1)$$ for $n \geq 1$ and $n$ odd.
\end{theorem}

\begin{proof}
We write a linear basis for $\Der(\land V)$.  The following table groups the basis elements for $\Der(\land V)$ according as they contribute to the homology of $\Der(\land V)$.  We shall see that the elements under group $0$ are those that persist to homology, whereas all the remaining groups of terms form short exact sequences that do not contribute to homology.  The result will follow. 

\medskip

\begin{center}
\begin{tabular}{c|c|c|c|c|c}
degree & group $0$ & group $1$ & group $2$ & group $3$ & more groups\\
\hline
$4n$ & $y^*$ & & &\\
\hline
$3n$ & & $u_1^*$, $u_2^*$ & & \\
\hline
$3n-1$ & & $(y, v_1)$, $(y, v_2)$ & &\\
\hline
$2n$ & $(y,w)$ & &  $w^*$ & \\
\hline
$2n-1$ & & &  $(u_1, v_2)$, $(u_2, v_1)$ &$(u_1, v_1)$,  $(u_2, v_2)$ \\
\hline
$2n-2$ & & &  $(y, v_1v_2)$ & $(y, v_1^2)$,  $(y, v_2^2)$ \\
\hline
$n+1$ & & &  &  & \vdots \\
\hline
$n$ & & &  &  & \vdots \\
\hline
$n-1$ & & &   &  & extracted \\
\hline
$n-2$ & & &  &  & below \\
\hline
$n-3$ & & &  &  & \vdots \\
\end{tabular}
\end{center}

\medskip

The lower-right portion of the table is as follows:

\medskip

\begin{center}
\begin{tabular}{c|c|c|c|c}
degree & group $4$ & group $5$ &  group $6$ & more groups\\
\hline
$n+1$ & $v_1^*$ &  $v_2^*$ &   \\
\hline
$n$ &  $(y, u_1)$, $(u_2, w)$ &  $(y, u_2)$, $(u_1, w)$ &   \\
\hline
$n-1$ &  $(y, v_1w)$ &  $(y, v_2w)$ & $(w, v_1)$ & \vdots \\
\hline
$n-2$ &  &  &  $(u_1, v_1v_2)$, $(u_2, v_1^2)$  & below \\
\hline
$n-3$ &   &  &  $(y, v_1^2v_2)$    & \vdots \\
\end{tabular}
\end{center}

\medskip

And the lower-right portion of this table is as follows:

\medskip

\begin{center}
\begin{tabular}{c|c|c}
degree &  group $7$ & group $8$ \\
\hline
$n-1$ &    $(w, v_2)$ \\
\hline
$n-2$   & $(u_1, v_2^2)$, $(u_2, v_1v_2)$ & $(u_1, v_1^2)$, $(u_2, v_2^2)$ \\
\hline
$n-3$  & $(y, v_1v_2^2)$& $(y, v_1^3)$, $(y, v_2^3)$\\
\end{tabular}
\end{center}

\medskip

\noindent\textit{Group 0}: It is clear that $y^*$ and $(y, w)$ are both non-bounding $D$-cycles.

\medskip

\noindent\textit{Group 1}:   We have
$$D(u_1^*) =  (y, v_1), \qquad D(u_2^*) =  (y, v_2).$$
Hence $D \colon \Der^{3n}( \land V) \to \Der^{3n-1}(\land V)$ is a linear isomorphism.  

\medskip

\noindent\textit{Group 2}:   We have $D(w^*) = d \circ w^*  - w^* \circ d = - w^* \circ d$.  When this is evaluated on $u_1$ and $u_2$, which are the only elements whose differentials involve $w$, we find that
$$D(w^*) =  (u_1, v_2) -  (u_2, v_1).$$
Furthermore, we have
$$D\big((u_1, v_2)\big) =  (y, v_1v_2), \qquad  D\big( (u_2, v_1)\big)= (y, v_1v_2).$$
It follows that 
$$0 \to \langle w^* \rangle \to \langle (u_1, v_2), (u_2, v_1) \rangle\to \langle (y, v_1v_2) \rangle \to 0,$$
in which the maps are $D$, is a short exact sequence.  Hence the group 2 terms contribute no homology.  

\medskip

\noindent\textit{Group 3}:   We have 
$$D\big((u_1, v_1)\big) = (y, v_1^2), \qquad D\big((u_2, v_2)\big) = (y, v_2^2).$$
Hence $D$ gives a linear isomorphism 
$$\langle (u_1, v_1), (u_2, v_2) \rangle \to \langle (y, v_1^2), (y,  v_2^2) \rangle,$$
and the group 3 terms contribute no homology. 

\medskip

\noindent\textit{Groups 4, 5, 6 and 7}:   These are shown to contribute no homology in the same way as for the group 2 terms: we have a short exact sequence in each case.

\medskip

\noindent\textit{Group 8}:  This is shown to contribute no homology in the same way as for the group 3 terms: $D$ gives a  linear isomorphism.

Thus far, we have shown that $\pi_i(\aut (X_n)) \cong H_i(\Der(\land V)) \cong \Q$ for $i = 2n, 4n$, and is zero otherwise.  This gives $B \aut (X_n)$ the correct rational homotopy groups, but in fact this is sufficient to determine the rational homotopy type, since there is a unique rational homotopy type with the desired rational homotopy groups.  This completes the argument for $n \geq 5$.  We briefly indicate how things proceed in the low-end cases. 

\noindent$n=3$:  Referring to the minimal model,  we have  $4n = 3n+3 = 12$, and so the highest-degree generator $y$ is in the same degree as the cubic terms in the $v_i$.  In the groups that appear in $\Der(\land V)$, groups $1$--$5$ are unchanged.  The only point to bear in mind for the remaining three groups is that $n-3 =0$.  However, the outcome, as regards homology in positive degrees---which is what we are concerned with here---is unchanged. Namely, these zero-degree terms still play the role of being in the image of the differential $D$ from degree $1$, meaning that we still have no non-zero (positive-degree) homology from these groups of terms.

\noindent$n=1$:  Here there is more coalescing of degrees of the various terms.   Referring to the minimal model,  we have  $2n = n+1 = 2$, and so the generators $w, v_1, v_2$ are now all in the same degree of $2$.  Furthermore, $3n+3 = 6$, whereas $4n = 4$,  so the cubic terms in the $v_i$ now appear above the highest-degree generator $y$, and so may be set aside.  Also, the quadratic terms in the $v_i$, as well as the products $wv_i$ appear in the same degree as $y$.  In the groups of terms that appear in $\Der(\land V)$, groups $1$--$5$ are unchanged.  As in the previous case, although here we have $n-1 =0$, the terms that appear in this degree still play their same role, as boundaries of elements from degree $1$, which means that the positive-degree homology contributed by those terms is still zero.  Here, groups $6$--$8$ may be set aside, as they occur completely in non-positive degrees (their homology is still zero, though). 
\end{proof}

\section{Non-Realization Results}\label{sec: non-realization}

In this section, we prove that several simple rational homotopy types cannot be realized as $\B (X)$ for $X$ simply connected and $\pi$-finite.  We begin with the following: 

\begin{theorem}\label{thm:CP}
There is no simply connected, $\pi$-finite, rational space $X$ for which $B \aut (X)$ has the rational homotopy type of $\C P^2$.   
\end{theorem}

\begin{proof}
In fact we show a  more general statement.  We will assume only that $B \aut (X)$ has rational homotopy groups of the form
$$\pi_i(B \aut (X)) = \begin{cases}  0 & i\geq 5 \\
\Q & i = 5 \\
0 & i = 3, 4\\
\Q^k\  \mathrm{some}\  k \geq 1 & i = 2
\end{cases}$$
and conclude that, at least  if $X$ is assumed $\pi$-finite and simply connected, $B \aut (X)$ must have infinite rational category. This rules out the possibility  of
  $B \aut (X)$ being the rationalization of $\C P^2$ or   of any  other finite complex.

To this end, we first show that, without loss of generality, we may assume that $X$ has minimal model of the form
\begin{center}
\begin{tabular}{c|c}
degree & generators\\
\hline
4 & $y$\\
\hline
3 & $u_1, \ldots, u_r$\\
\hline
2 & $v_1, \ldots, v_r$\\
\end{tabular}
\end{center}
for some $r \geq 2$, and furthermore that the differential on the top-degree generator is
$$dy = u_1 v_1 + \cdots + u_rv_r.$$
There may be non-zero differentials $d \colon V^3 \to \land^2 (V^2)$ as well, but these do not play a role in our argument. 

To see this,  start by applying \lemref{lem: pi-finite top degree} to obtain that the minimal model for $X$ must have a single generator in degree $4$, and no higher-degree generators.  Then write $V^3 = \langle u_1, \ldots, u_r \rangle$ and  $V^2 = \langle v_1, \ldots, v_s \rangle$, for some $r, s \geq 0$.  Now we must have $r \geq 1$, otherwise there would be no possibility for having $\pi_2(B \aut (X))  \not= 0$.  For degree reasons, we may write 
$$d(y) = u_1\beta_1 + \cdots + u_r\beta_r$$
for some $\beta_i \in V^2$ (possibly zero, at this point in the argument).  Then the derivations $u_i^*$ have boundary $D(u_i^*) = (y, \beta_i)$ for each $i$.  So, if $\{ \beta_1, \dots, \beta_r \}$ were linearly dependent, we would have a cycle of degree $3$, of the form  $\sum u_i^*$, that could not be a boundary, and so would contribute a non-zero element to  $\pi_4(B \aut (X))\otimes \Q$.  This contradicts our assumption on the rational homotopy of $B \aut (X)$, and so we must have $s \geq r$, with $\{ \beta_i \}_{i=1, \ldots, r}$ linearly independent in $V^2$.  Next, consider the derivations $(y, v_j)$ of degree $2$.  Each of these is a cycle and, since we assume $\pi_3(B \aut (X)) = 0$, the image of the differential $D$ must span $\langle (y, v_1), \ldots, (y, v_s) \rangle$.  However, the only boundaries we have available here are given by the $D(u_i^*)$, since the only degree-$3$ derivations are the $u_i^*$.  Therefore, we must have $r \geq s$, and hence $r = s$.  Then the $\{ \beta_i \}_{i=1, \ldots, r}$ are a basis for $V^2$, and we re-label them as $\beta_i = v_i$ for each $i$.   

Next we  eliminate the case in which $r = 1$.  This case consists of the minimal model $(\land(v, u, y), d)$ with single non-differential $d(y) = uv$ (it is not possible for $d(u)$ to be non-zero here).  A direct calculation shows that, for this $X$, we have $B \aut (X) = S^5$, and in particular $\pi_2(B \aut (X)) = 0$.

Now suppose $X$ has minimal model $(\land(v_1, \ldots, v_r, u_1, \ldots, u_r, y), d)$, with the degrees as above, with $r\geq 2$, and with the differential on $y$ of the form
$$d(y) = u_1v_1 + u_2v_2 + \cdots + u_rv_r.$$
 We  construct a KS-extension
$$(\land (z_2), 0) \to (\land(z) \otimes \land(v_1, \ldots, v_r, u_1, \ldots, u_r, y ), \D) \to (\land(v_1, \ldots, v_r, u_1, \ldots, u_r, y), d),$$
by setting $\D(u_1) = z v_2 + d(u_1)$ and $\D(u_2) = z v_1 + d(u_2)$, and $\D = d$ on all generators other than $u_1, u_2$.  We see directly that $(\land(z) \otimes \land(v_1, \ldots, v_r, u_1, \ldots, u_r, y ), \D)$ is minimal.  Using the fact that   $D(z^*) = \pm (u_1, v_2) \pm (u_2, v_1)$ is non-zero it is easy to see that $\theta(z) \neq 0$ implies $D(\theta) \neq 0.$   Proposition \ref{lem:non-trivial} gives an essential map  $h \colon K(\Q, 2) \to B \aut (X).$
   Since we assume $B \aut (X)$ only has non-zero rational homotopy groups in degrees $2$ and $5$, the only possibility for such a  map is one that is injective in rational homotopy groups in degree $2$.   This implies, using the mapping theorem of F{\'e}lix-Halperin \cite[Th.28.6]{F-H-T01}, that $\cat0 ( B \aut (X) )= \infty$.  In fact, it is easy to see that $H^2(B \aut (X))$ must contain an element $a$ such that $a^n \not=0$ for all $n \geq 1$, so that $B \aut (X)$ actually has infinite rational cup-length.    
\end{proof}

We apply  similar  arguments to prove:

\begin{theorem}\label{thm: S4}
There is no  simply connected, $\pi$-finite, rational space $X$ for which $B \aut (X)$ has the rational homotopy type of $S^4$.   
\end{theorem}

\begin{proof}  Let  $\iota_4 \in \pi_4(S^4) $ denote the fundamental class with nontrivial  Whitehead product $[\iota_4, \iota_4] \in \pi_7(S^4)  $.      Suppose given $X$ with $B\aut (X) = S^4$     and  minimal model $\land ( V, d).$ Then we have  
$V^6 = \langle y_6 \rangle$ by Proposition \ref{lem: pi-finite top degree}.   Further,  the  derivation cycle $y^*$     must decompose as $y^* = [\theta, \theta]$ for $\theta$ a degree 3 derivation cycle.      In particular,  $\dim V^3 > 0$.  Let $x \in V^3.$   

Suppose   $V^2 = 0$.   Then  $V^5 =0$ for otherwise, if $u \in V^5$,   the  degree  $5$ derivation cycle $u^*$  
 cannot  bound.   Also,  $dx = 0$ and so  $(y, x)$ is a degree $3$     cycle that does not  bound.  We cannot then have that  $x^*$ is a cycle,    since $x^*$  cannot bound as this would produce too many homology elements in degree 3.  We conclude there is  an element $w_4 \in V^4$ and $dy = wx + \hbox{other terms}$.  For degree reasons $dw = 0.$ We   construct  a KS-extension 
 $  (\land(z_2),0)  \to (\land(z) \otimes  \land V, \D) \to  (\land (V), d)  $
 where $Dw = zx$ and $Dy = dy.$    Applying  Proposition  \ref{lem:non-trivial} gives an essential map  $h \colon K(\Q, 2) \to B\aut(X)$  which contradicts the assumption that $B\aut(X) = S^4$.   We conclude $\mathrm{dim} V^2 > 0.$ 

Now suppose $\dim V^2 = 1$. Then $\dim V^5 = 1$ also and $dy = v_2u_5 + $ other terms.  
Suppose $V^4 \neq 0.$   Then the degree $2$ cycle $(y, w_4)$ must be a boundary which forces $dy$ to take the form  $dy = vu + wx $ + other terms. 
Now   $w$ cannot appear in  $dv$ or else $d^2y \neq 0.$  We can thus construct the same KS-extension as above to derive a contradiction.    We conclude that   $V^4 = 0$.  The minimal model for $X$ must then be of  the form 
$(\land(v_2, x_3, u_5, y_6),  d)$.  
There are two cases to check here.    If $dx =  v^2$ then $x$ cannot appear in $dy$ and so we may write $dy = uv + q v^3$ for $ q \in \Q$, possibly zero.  In this case, $x^*$ and $(v, u) - (y, x)$ are cycles of degree $3$.   There are no boundaries in degree 3 and so this is a contradiction.  The other possibility is that $dx = 0$.   Then $x$ must appear in the differential $dy$ to ensure $x^*$ and $(y, x)$ 
do not give too many degree $3$ derivation cycles.  So $dy = uv +v^2x + qv^3$ is the only non-trivial differential.  It is now easy to compute
that $B \aut (X) \simeq_\Q K(\Q, 2) \times K(\Q, 4) \times K(\Q, 7).$ 

It remains to handle the cases where $\dim V^2 > 1.$  Write $V^2 = \Q(v_1, \ldots, v_r)$ for $ r\geq 2.$   Then the derivations cycles  $(y, v_i)$ must each bound.  This forces   $V^5 = \Q(u_1, \ldots, u_r)$ and $$dy = u_1v_1 + u_2v_2 + \cdots + u_rv_r + \hbox{other terms}.$$
We now apply the same argument used in the  proof of Theorem \ref{thm:CP}.  Specifically,  we obtain a KS-extension
$  (\land (z_4), 0) \to (\land(z) \otimes \land V, \D) \to (\land( V), d) $ 
satisfying the conditions in  Proposition \ref{lem:non-trivial} and so giving an essential map $h \colon K(\Q, 4) \to B\aut (X).$ Since $S^4$ admits no such map, the proof is complete.  
 \end{proof}

Our last result uses the notion of \emph{positive weights} on a minimal model.  This notion has its origins in work of  Body-Douglas \cite{B-D80} and 
 Mimura-O'Neill-Toda \cite{MOT} on $p$-universal spaces.  

\begin{definition}We say that a DG algebra $(A, d)$ has a \emph{positive weight decomposition} if it admits a direct sum decomposition $A^{+} = \oplus_{i \geq 1} A^{+}(i)$ that satisfies $A(i) \cdot A(j) \subseteq A(i+j)$ and $d\big( A(i) \big) \subseteq A(i)$.  
\end{definition}

We say that a space $X$ has positive weights, or is $p$-universal, if some model for it admits a positive weight decomposition.   Notice that the property is independent of any particular type of model.  If either a DG algebra (Sullivan) model, or a DG Lie algebra (Quilllen) model for $X$ admits a positive weight decomposition, then this may be translated into the existence of a family of self-maps of $X$, corresponding  to grading automorphisms of the model that admits the weight decomposition.  In this way, the condition may actually be phrased purely in terms of self-maps of the space $X$, independently of any choice of model.  Indeed, the notion of $p$-universality actually pre-dates rational homotopy theory and minimal models.    

We begin with the following observation.

\begin{lemma}\label{lem: shallow pi-finite}
Suppose $X$ is a simply connected, $\pi$-finite, rational  space with $\pi_i(X) = 0$ unless $i = 2, 3, 4$.  Then the DG Lie algebra $\Der(\land V)$ admits a positive weight decomposition.  Consequently, $B \aut (X)$ is a $p$-universal space.
\end{lemma}

\begin{proof}
Suppose $X$ has minimal model of form $\land V = \land (V_2, V_3, V_4)$, with $V_i$ the vector spec of generators of degree $i$.  The differential $d$ in $\land V$ satisfies $d(V_2) = 0$, $d(V_3) \subseteq \land^2 V_2$, and $d(V_4) \subseteq V_2 \cdot V_3$.  In the stye of the above examples, we may write a basis for $\Der(\land V)$ as follows.

\begin{center}
\begin{tabular}{c|c|}
degree & generators\\
\hline
4 & $V_4^*$\\
\hline
3 & $V_3^*$ \\
\hline
2 & $V_2^*$, $(V_4, V_2)$\\
\hline
1 & $(V_3, V_2)$, $(V_4, V_3)$\\
\end{tabular}
\end{center}
 
Here, notation such as $(V_4, V_3)$ denotes $\Hom(V_4, V_3)$, with typical basis element $(x, y)$, where $x$ and $y$ are basis elements of $V_4$ and $V_3$ respectively. 
 Translating the differential from $\land V$ into that on $\Der(\land V)$, we see  that the differential $D$  in $\Der(\land V)$  satisfies $D(V_4^*) = 0$, and
$$ D\big( V_3^* \big) \subseteq  (V_4, V_2), \qquad D\big( (V_4, V_2 \big)  = 0, $$
$$D\big( V_2^* \big) \subseteq  (V_3, V_2) \oplus (V_4, V_3).$$
Furthermore, the only possible non-zero brackets in $\Der(\land V)$ satisfy
$$ [ (V_3, V_2) ,  (V_4, V_3)] \subseteq (V_4, V_2), \qquad  [ V_2^* ,  (V_3, V_2)] \subseteq V_3^*, $$
$$ [ V_3^* ,  (V_4, V_3)] \subseteq V_4^*), \qquad  [ V_2^* ,  (V_4, V_2)] \subseteq V_4^*, $$
It follows that, if we assign  positive weights of $1$ to  $V_2^* \oplus (V_3, V_2) \oplus (V_4, V_3)$, $2$ to $V_3^* \oplus  (V_4, V_2)$, and $3$ to $V_4^*$, then we have a positive weight decomposition on $\Der(\land V)$, which is a DG Lie algebra model for $B \aut (X)$.  
\end{proof}

Next, we give an example of a space that is not $p$-universal, whose rational homotopy groups are concentrated in degree $2, 3, 4, 5$.  

\begin{theorem}\label{ex: non-universal}
Suppose $Y$ is the space with Sullivan minimal model 
$$(\land( a_2, b_2, c_2, x_3, y_3, z_3, \phi_4, \psi_4, w_5),  d),$$
where subscripts denote degrees, and the differential, where non-zero,  is given by 
$$ d(x) = a^2 + ac, \qquad d(y) = ab, \qquad d(z) = bc,$$
$$d(\phi) = xb - ay - az, \qquad d(\psi) = cy - az,$$
$$d(w) = \phi a + xy + \psi a + c^3 + b^3.$$
 Then $Y$ cannot be of the rational homotopy type of $B \aut (X)$, for any $X$ a simply connected, $\pi$-finite, rational  space.  
\end{theorem}
\begin{proof}
We claim    $Y$ is not $p$-universal.  For suppose the above  model admits a positive weight decomposition. Write the weight of an element $\chi \in \land V$ as $\w(\chi)$.   Since the boundary $a^2 + ac$ must be of homogeneous weight, it follows that $\w(a) = \w(c)$.  Likewise, since the boundary $\phi a + xy + \psi a + c^3 + b^3$ must be of homogeneous weight, it follows that $\w(b) = \w(c)$.  Thus we have $\w(a) = \w(b) = \w(c) = r \geq 1$, say.  From the formulas for their differentials, then, we have $\w(y) = \w(z) = 2r$, and $\w(\phi) = \w(\psi) = 3r$.  Finally, the boundary      
$\phi a + xy + \psi a + c^3 + b^3$ is not of homogeneous weight, since the first three terms have weight $4$, whilst the last two have weight $3$.  This is a contradiction. 

Now suppose $Y = B\aut (X)$ for some $\pi$-finite, simply connected space $X$.  
Then, by Proposition \ref{lem: pi-finite top degree},  $X$ would have generators concentrated in degrees $2, 3, 4$.  Indeed, $X$ would have to have $\pi_4(X)  $ of rank-one.   In \lemref{lem: shallow pi-finite}, we showed that $B \aut (X)$ for such a space is $p$-universal.  Since $Y$ is not such, it cannot be obtained as $B \aut (X)$.
\end{proof}

We conclude by adding  the following related observation:

\begin{theorem}\label{thm: formal Baut weighted}
Suppose $X$ is a formal space.  Then $B \aut (X)$ has positive weights (is $p$-universal).
\end{theorem}

\begin{proof}
For this argument, we need the alternative  DG Lie algebra model for the classifying space $B \aut(X)$ expressed in terms of derivations of the Quillen model   $L = \mathbb{L}(V;d)$  for $X$.  This  is a  DG Lie algebra of the form $(sL \oplus \Der L, D)$ where $sL$ is the graded suspension of $L$ and $\Der L$ the Lie algebra of degree lowering derivations of $L$. We refer the reader to \cite[Ch.6]{Tanre}  for the details of this construction.  We here observe that  $(sL \oplus \Der L, D)$ may be given a positive weight decomposition.

Start with a standard positive weight decomposition of $L = \mathbb{L}(V;d)$.  Namely, for an element $x \in L$ of homogeneous length and degree, assign $x$ a weight equal to the sum of its degree and length, thus
$$\w(x) = |x| + l(x).$$
Since $X$ is formal, we may assume that the differential $d$ in $L$ is quadratic, so increases length by $1$.  On the other hand, $d$ decreases degree by $1$, and hence this choice of weighting is preserved by $d$.  Evidently, this choice of weighting respects brackets of elements, too, and so it gives a positive weight decomposition to the Quillen minimal model $L$.

Then, assign weights to elements of  $(sL \oplus \Der L, D)$ as follows.  Suppose that   $\{ v_i \}$ is a basis of $V$, and $\{ \chi_j \}$ is  a basis for $L = \mathbb{L}(V)$ that is homogeneous with respect to degree and length (e.g. a standard Hall basis).  Then the derivations $\{ (v_i, \chi_j) \}$ give a basis for $\Der L$.    Now set $\w( sx) = \w(x)$ for any homogeneous weight element $x \in L$, and $\w\big( (v_i, \chi_j)\big) = \w(\chi_j) - \w(v_i)$ for each $i, j$.  Since we are restricting to positive-degree derivations in $\Der L$, we have
$$\w\big( (v_i, \chi_j)\big) = \w(\chi_j) - \w(v_i) = |\chi_j| - |v_i| + l(\chi_j) - 1 > 0,$$
since $|\chi_j| - |v_i| > 0$ for any positive-degree derivation, and $l(\chi_j) - 1 \geq 0$.  It remains to check that brackets and the differential $D$ behave well with respect to this weighting.

Recall that brackets $(sL \oplus \Der L, D)$ are given by the usual bracket of derivations amongst elements of $\Der L$, whilst brackets amongst  elements of $sL$ are trivial, and brackets ``across" $sL$ and $\Der L$ are given by $[\theta, sx] = (-1)^{|\theta|} s\theta(x)$.  It is straightforward to check that weights as we have assigned them add under these brackets.  The differential in $(sL \oplus \Der L, D)$ is the usual $D = \ad(d)$ on derivations in $\Der L$.  Since $d$ preserves weight (of elements in $L$), it is easy to see that $D$ preserves weight (of derivations).  Finally, for elements $sx \in sL$, the differential is defined as $d(sx) = - sdx + \ad(x)$.  For a homogenous weight $x$, we have assigned $sx$ the weight of $x$, which is the same as the weight of $dx$, and also the weight of $\ad(x)$ (as a derivation).  It follows that $D$ preserves the weight of elements $sx$ also.
\end{proof}


\bibliographystyle{amsplain}
\bibliography{Baut}

\end{document}